\documentclass[10pt]{article}
\font\smallit=cmti10

\usepackage{amssymb,latexsym,amsmath,epsfig,amsthm,cite,algorithm,algorithmic,hyperref} 

\makeatletter

\renewcommand\section{\@startsection {section}{1}{\z@}
{-30pt \@plus -1ex \@minus -.2ex}
{2.3ex \@plus.2ex}
{\normalfont\normalsize\bfseries}}

\renewcommand\subsection{\@startsection{subsection}{2}{\z@}
{-3.25ex\@plus -1ex \@minus -.2ex}
{1.5ex \@plus .2ex}
{\normalfont\normalsize\bfseries}}

\renewcommand{\@seccntformat}[1]{\csname the#1\endcsname. }

\makeatother

\newtheorem{theorem}{Theorem}
\newtheorem{lemma}{Lemma}

\begin{document}

\begin{center}
\uppercase{\bf Robin's inequality for new families of integers}
\vskip 20pt
{\bf Alexander Hertlein}\\
{\smallit Ludwig-Maximilians-Universit\"at, 80539 M\"unchen, Germany}\\
\end{center}
\vskip 30pt
\centerline{} 
\vskip 30pt

\centerline{\bf Abstract}
\noindent Robin's criterion states that the Riemann Hypothesis is true if and only if Robin's inequality $\sigma(n):=\sum_{d|n}d<e^{\gamma} n \log \log n$ 
is satisfied for each $n > 5040$, where $\gamma$ denotes the Euler-Mascheroni constant.
We show that if a positive integer $n$ satisfies either $\nu_2(n) \leq 19$, $\nu_3(n) \leq 12$, $\nu_5(n) \leq 7$, $\nu_7(n) \leq 6$ or
$\nu_{11}(n) \leq 5$ then Robin's inequality is satisfied, where $\nu_p(n)$ is the p-adic order of $n$.
In the end we show that $\sigma(n)/n < 1.0000005645 \ e^{\gamma}\log \log n$ holds unconditionally for $n > 5040$.

\pagestyle{myheadings} 
\thispagestyle{empty} 
\baselineskip=12.875pt 
\vskip 30pt

\section{Introduction}
Let $n$ be an integer satisfying $\sigma(n):=\sum_{d|n}d<e^{\gamma} n \log \log n$, where $\gamma$ denotes the Euler-Mascheroni constant. This 
inequality is called {\it Robin's inequality}. Robin \cite{R} proved that the Riemann Hypothesis (RH) is true if and only if his inequality holds for every integer $n > 5040$.
So far Robin's inequality has been proven unconditionally for families of integers that are
\begin{itemize}
 \item odd and greater than $9$ \cite{C}
 \item square-free and greater than $30$ \cite{C}
 \item a sum of two squares and greater than $720$ \cite{BM}
 \item not divisible by the fifth power of a prime \cite{C}
 \item not divisible by the seventh power of a prime \cite{PS}
 \item not divisible by the eleventh power of a prime \cite{BT}.
 \end{itemize}
Here, we extend Robin's inequality.
We first provide a modified algorithm of the one obtained by Akbary et al. \cite{A} to establish the exceptions to the inequality $n/\varphi(n) < (1771561/1771560) e^{\gamma} \log \log n$, where $\varphi$ stands for Euler's totient function.
With this we then show that if $n$ has a 2-adic order smaller or equal to 19 or satisfies either $\nu_3(n) \leq 12$, $\nu_5(n) \leq 7$, $\nu_7(n) \leq 6$ or $\nu_{11}(n) \leq 5$ then Robin's inequality holds. 
Then we find that $\sigma(n)/n < 1.0000005645 \ e^{\gamma}\log \log n$ holds unconditionally for all $n > 5040$.
\section{Theorems}
We first want to show the case where we know that the 2-adic order of $n$ is lower or equal to 19.
\begin{theorem}
\label{evennumber}
Robin's inequality holds for $n > 5040$ when $\nu_2(n) \leq 19$.
\end{theorem} 
We then go on to partially prove a result of Choie et. al \cite{C}.
\begin{theorem}
\label{5thprime}
Consider those integers $n$ which satisfy $\nu_3(n) \leq 12$, $\nu_5(n) \leq 7$, $\nu_7(n) \leq 6$ or
$\nu_{11}(n) \leq 5$. Then, Robin's inequality holds for all such integers $n > 5040$.
\end{theorem}
An improved unconditional upper bound of $\sigma(n)/n$ is provided by the following.
\begin{theorem}
\label{unconditionalbound}  
The inequality
\begin{equation}
\sigma(n)/n < 1.0000005645 \ e^{\gamma}\log \log n
\end{equation}
holds for all $n > 5040$.
\end{theorem} 
\section{Proofs}
\begin{lemma}
\label{phiandsigma}
Let $\displaystyle\prod_{i = 1}^{r} q_i^{a_i}$ be the representation of n as a product of primes $q_1 < ... < q_r$ with positive exponents $a_1 < ... < a_r$. Then
\begin{equation}
\frac{\sigma(n)}{n}=\frac{n}{\varphi(n)} \displaystyle\prod_{i = 1}^{r}\left(1 - \frac{1}{q_i^{a_i+1}}\right). 
\end{equation}
\end{lemma}
\begin{proof}
This is Lemma 2 in \cite{G}. \end{proof}
We now take a look at a way to establish a new upper bound for $n/\varphi(n)$.
First we provide an algorithm which is derived from Akbary et al. \cite{A}.
\label{newalgo}
They developed an algorithm that calculates the exceptions to the following inequality where $0 < \epsilon < 1$ and $\omega(n)$ is the number of distinct prime divisors of $n$:
\begin{equation}
\label{inequalityfn}
f(n) := \prod_{\substack{p\leq p_{\omega(n)} \\ p \ prime}} \frac{p}{p-1} < e^\gamma (1+\epsilon) \log\log n.
\end{equation}
For an integer $n$ and an integer $\beta \geq \omega(n)\geq 2$ they showed that if 
\begin{equation}
n > n_\beta := \exp\left(\exp\left(\frac{1}{(1+\epsilon)e^{\gamma}} \prod_{p \leq p_\beta}\frac{p}{p-1}\right)\right)
\end{equation}
then inequality (\ref{inequalityfn}) is satisfied. According to Lemma 3.4 in \cite{A}, we only need to find the first $\beta$ for a given $\epsilon$ for which $\displaystyle \prod_{p \leq p_\beta}p < n_\beta$ does not hold in order to get to the largest possible exception of (\ref{inequalityfn}). We call this largest possible exception of inequality (\ref{inequalityfn}) $n_{\beta_{max}}$. We can now describe the modified algorithm which is proven to be correct by Lemma 3.4 in \cite{A}.
\begin{algorithm}
\caption{Largest possible exception to $f(n) < e^\gamma (1+\epsilon) \log\log n$}
\begin{algorithmic}             
    \REQUIRE $0<\epsilon<1$
    \ENSURE Largest possible exception to the inequality.
    \WHILE{$\displaystyle \prod_{p \leq p_\beta}p < n_\beta$}
        \STATE{$\beta \rightarrow \beta +1$}
    \ENDWHILE
    \\$\beta_{max} \rightarrow \beta$
    \\$n_{\beta_{max}} \rightarrow n_\beta$
\end{algorithmic}
\end{algorithm}\\
We can now go on to find an upper bound for $n/\varphi(n)$.
\begin{lemma}
\label{phiupperbound}
The inequality
\begin{equation}
\label{inequalityphiupper}
\frac{n}{\varphi(n)} < \frac{1771561}{1771560} e^{\gamma} \log \log n
\end{equation}
is satisfied for all $n > c_0 := e^{e^{23.762143}} $.
\end{lemma}
\begin{proof}
On noting that,
\begin{equation}
\label{phiupper}
\frac{n}{\varphi(n)} \leq \prod_{p \leq p_\beta}\frac{p}{p-1} < e^\gamma (1+\epsilon) \log\log n
\end{equation}
we run the algorithm from Lemma \ref{inequalityfn} with $\epsilon = 1/1771560$ such that the RHS of (\ref{phiupper}) matches the RHS of (\ref{inequalityphiupper}).
The result of the algorithm, namely $\beta_{max}$ and $n_{\beta_{max}}$ is
$$\beta_{max} = 919356257 \quad \quad n_{\beta_{max}} < e^{e^{23.762143}} $$
We note that $n_{\beta_{max}}$ cannot be exactly numerically calculated to integer precision, which is mainly due to the sheer size of the number. Fortunately, this is not necessary, since we can bound $n_{\beta_{max}}$ from above in our numerical calculation and still maintain the correctness of the algorithm. This is why we limit the numerical computation of the exponent of $n_{\beta_{max}}$ to 200 digits and then use the exponent $23.762143$. 
Since this calculated bound is important throughout our proofs we set $c_0 := e^{e^{23.762143}}$.

The algorithm guarantees that all exceptions to inequality (\ref{inequalityfn}) are below $c_0$, which allows us to conclude that for all $n > c_0$ the inequality (\ref{inequalityphiupper}) holds. 
\end{proof}
\begin{lemma} 
\label{Briggs}
Robin's inequality is true for all $5040 < n \leq 10^{10^{10}}$.
\end{lemma}
\begin{proof}
Robin showed in \cite{R}, Prop.1, p.192 that if Robin's inequality holds for consecutive colossally abundant numbers $n_1$ and $n_2$ then it also holds for all $n \in [n_1,n_2]$. By definition an integer $n$ is colossally abundant if there exists a positive $\epsilon$ for which $\sigma(n)/n^{1+\epsilon} \geq \sigma(k)/k^{1+\epsilon} $ for all $k > 1$. Briggs \cite{B} showed that Robin's inequality holds for all colossally abundant numbers between $5040$ and $10^{10^{10}}$. We may therefore conclude that Robin's inequality is also satisfied for all integers $5040 < n < 10^{10^{10}}$.\end{proof}
We are now ready to prove Theorem \ref{evennumber}.\\ 
\textbf{Proof of Theorem \ref{evennumber}}
\begin{proof}
We now let $n$ have a 2-adic order of $\nu_2(n) \leq 19$.
From Lemma \ref{phiandsigma} we note that 
\begin{equation}
\frac{\sigma(n)}{n}=\frac{n}{\varphi(n)} \displaystyle\prod_{i = 1}^{r}\left(1 - \frac{1}{q_i^{a_i+1}}\right) \leq \frac{n}{\varphi(n)} \left(1 - \frac{1}{2^{\nu_2(n)+1}}\right).
\end{equation}
We only need to look at the case where $\nu_2(n) = 19$ since the weaker cases follow because $$\left(1 - \frac{1}{2^{1+1}}\right) < \left(1 - \frac{1}{2^{1+2}}\right)< ... <\left(1 - \frac{1}{2^{1+19}}\right).$$
With Lemma \ref{phiupperbound} we have for $n > c_0$ \\
\begin{equation}
\begin{aligned}
\frac{\sigma(n)}{n} \stackrel{\nu_{2}(n) = 19}{\leq} \frac{n}{\varphi(n)} \left(1 - \frac{1}{2^{1+19}}\right) = \frac{1048575}{1048576}\frac{n}{\varphi(n)}\\<\frac{1048575}{1048576}\frac{1771561}{1771560} e^{\gamma} \log \log n < e^{\gamma}\log \log n.\end{aligned}
\end{equation}
In light of Lemma \ref{Briggs} and the fact that $c_0 < 10^{10^{10}}$ we then conclude that Robin's inequality is true for those $n > 5040$ for which $\nu_2(n) \leq 19$.
\end{proof}
Our proof of Theorem \ref{5thprime} is now done with other p-adic orders used to partially prove Theorem 6 of \cite{C}. \\
\textbf{Proof of Theorem \ref{5thprime}}
\begin{proof}
We now consider $n$ with an 11-adic order satisfying $\nu_{11}(n) \leq 5$. The cases for the 3-adic, 5-adic or 7-adic order follow directly since
$$\left(1 - \frac{1}{5^{1+7}}\right)< \left(1 - \frac{1}{7^{1+6}}\right) < \left(1 - \frac{1}{3^{1+12}}\right) < \left(1 - \frac{1}{11^{1+5}}\right).$$
With Lemma \ref{phiandsigma} and \ref{phiupperbound} we then have for $n > c_0 $\\
\begin{equation}
\begin{aligned}
\frac{\sigma(n)}{n} \stackrel{\nu_{11}(n) = 5}{\leq} \frac{n}{\varphi(n)} \left(1 - \frac{1}{11^{1+5}}\right) = \frac{1771560}{1771561}\frac{n}{\varphi(n)}\\< \frac{1771560}{1771561}\frac{1771561}{1771560} e^{\gamma} \log \log n = e^{\gamma}\log \log n.
\end{aligned}
\end{equation}
By invoking Lemma \ref{Briggs} and noting that $c_0 < 10^{10^{10}}$ we then conclude that Robin's inequality is true for those integers $n > 5040$ for which $\nu_3(n) \leq 12$, $\nu_5(n) \leq 7$, $\nu_7(n) \leq 6$ or $\nu_{11}(n) \leq 5$.
\end{proof}
With these results, we can now also improve the unconditional bound for $\sigma(n)/n$ from Akbary et al. \cite{A}. \\
\textbf{Proof of Theorem \ref{unconditionalbound}}
\begin{proof}
First, note that $1771561/1771560 = 1.000000\overline{564474248684775}.$ 
\\Then similar to Theorem 1, it follows from Lemma \ref{phiupperbound} that for $n > c_0$,\\
\begin{equation}
\frac{\sigma(n)}{n} \leq \frac{n}{\varphi(n)} < \frac{1771561}{1771560} e^{\gamma} \log \log n < 1.0000005645 \ e^{\gamma} \log \log n
\end{equation}
On invoking Lemma \ref{Briggs} we then find that the above inequality holds unconditionally for $n > 5040$.
\end{proof}
\section{Acknowledgement}
The author thanks Friedrich Hertlein, Felix Palm, Patrick Sol\'e, Lucas Kimmig and the referee for useful suggestions.
All calculations were done with Mathematica 10 on a 6-core Xeon processor with 200 digit precision.

\vfil\eject

\end{document}